\documentclass[oneside,10pt]{article}          
\usepackage[b5paper]{geometry}	    

\usepackage{graphicx}
\usepackage{epstopdf}
\usepackage{amsfonts,amsmath,latexsym,amssymb} 
\usepackage{amsthm}                
\usepackage{mathrsfs,upref}         
\usepackage{mathptmx}		    
\usepackage{mia}	            

\usepackage[final]
{showkeys}

\usepackage{enumerate}

\usepackage{hyperref}

\newtheorem{theorem}{Theorem}[section]           
\newtheorem{lemma}[theorem]{Lemma}               
\newtheorem{corollary}[theorem]{Corollary}

\newtheorem{conjecture}[theorem]{Conjecture}

\theoremstyle{definition}

\newtheorem{remark}{Remark}

\numberwithin{equation}{section}       

\newcommand{\de}{\delta}

\renewcommand{\Psi}{\overline{\Phi}}

\newcommand{\R}{\mathbb{R}}

\renewcommand{\L}{\mathcal{L}}

\renewcommand{\le}{\leqslant}
\renewcommand{\ge}{\geqslant}

\begin{document}

\title{One-sided concentration \\ near the mean of log-concave distributions}


\author{Iosif Pinelis}

\address{Department of Mathematical Sciences\\
Michigan Technological University\\
Houghton, Michigan 49931, USA\\
\email{ipinelis@mtu.edu}}

%

\CorrespondingAuthor{Iosif Pinelis}


\date{\today}                               

\keywords{log-concave distributions; probability inequalities; moments}

\subjclass{60E15, 26D10, 26D15, 26A51}


\begin{abstract}
A lower bound on the probability $P(0<X<\de)$ for all real $\de>0$ and all random variables $X$ with log-concave p.d.f.'s such that $EX=0$ and $EX^2=1$ is obtained.
\end{abstract}

\maketitle



\section{Summary and discussion}

Let $L$ denote the set of all real random variables (r.v.'s) $X$ with a log-concave p.d.f.\ $f_X$ such that $EX=0$ and $EX^2=1$. A broad survey on log-concavity was given by Saumard and Wellner \cite{wellner-logconcav}. 

\begin{theorem}\label{th:}
For all $X\in L$ and all real $\de>0$, 
\begin{equation}\label{eq:}
	P(0<X<\de)\ge p(\de):=\frac\de{72(1+c\de)},
\end{equation}
where $c:=\frac{419}{100}$. 
\end{theorem}

Letting $\de\downarrow0$ in Theorem~\ref{th:}, we immediately obtain

\begin{corollary}\label{cor:}
If $X\in L$, then $f_X(0)\ge\frac1{72}$.  
\end{corollary}

\bigskip

Somewhat related to these results are the following ones by    
Barlow, Karlin, and Proschan: if $X$ is a \emph{nonnegative} r.v.\ with a log-concave p.d.f.\ $f_X$ and $EX=1$, then 
\begin{enumerate}[(i)]
	\item $f_X(0)\le1$ (\cite[formula~(11)]{BKP}); 
	\item $f_X(x)\le e^{-x}$ for some (unspecified) real $x_0$ and all real $x>x_0$ (\cite[Theorem~5]{BKP}).
\end{enumerate}
The latter result implies the inequality $P(0<X<\de)\ge1-e^{-\de}$ for real $\de\ge x_0$. 

\bigskip

The lower bound $p(\de)$ on $P(0<X<\de)$ in \eqref{eq:} as well as the lower bound $\frac1{72}$ on $f_X(0)$ in Corollary~\ref{cor:} are likely rather far from optimality. 
In fact, as noted by the anonymous reviewer, \cite[Proposition B.2]{bobkov-ledoux19} provides a better lower on $f_X(0)$ than the one in Corollary~\ref{cor:}: $f_X(0)\ge\frac1{e\sqrt3}$. 
Further studies in this direction should be welcome. 

\begin{conjecture}\label{conj:}
Let $Y$ be a r.v.\ with the standard exponential distribution, so that $Y-1\in L$.  
Then for each real $\de>0$ 
\begin{equation*}
	\min_{X\in L}P(0<X<\de)=P(0<Y-1<\de)=e^{-1}(1-e^{-\de}), 	
\end{equation*}
and hence 
\begin{equation*}
	\min_{X\in L}f_X(0)=e^{-1}. 	
\end{equation*}
\end{conjecture} 

The target quantities $P(0<X<\de)$ and $f_X(0)$, as well as the  
quantities $EX$ and $EX^2$ in the condition $X\in L$, 
are linear in the p.d.f.\ $f:=f_X$. However, the condition that $f$ be log concave is transcribed as the system of uncountably many inequalities \break 
$f((1-t)x+ty)\ge f(x)^{1-t}f(y)^t$ for all real $x,y$ and all $t\in(0,1)$, and these inequalities are 
nonlinear in $f$. So, Conjecture~\ref{conj:} represents a problem of nonlinear infinite-dimensional optimization, which may be very nontrivial -- cf.\ e.g.\ \cite{Fattorini_1999}.

\section{Proof of Theorem~\ref{th:}}\label{proofs}

\begin{lemma}\label{lem:E|X|>1/2}
For all $X\in L$ we have $E|X|\ge1/2$. 
\end{lemma}

\begin{proof} Take any $X\in L$. Let $m$ be the median of $X$, so that $P(X<m)\break =P(X>m)=1/2$. Let $Z^\pm$ be r.v.'s whose respective distributions are the conditional distributions of $\pm(X-m)$ given $\pm(X-m)>0$. Then $Z^\pm$ are positive r.v.'s with log-concave densities. 
So, by a well-known inequality (see e.g.\ \cite[formula (0.3)]{keilson72}), 
\begin{equation*}
2(EZ^\pm)^2\ge E(Z^\pm)^2. \tag{45}\label{45} 	
\end{equation*}
Since
$EZ^\pm=2E(X-m)_\pm$ and $E(Z^\pm)^2=2E(X-m)_\pm^2$, where $u_\pm:=\max(0,\pm u)$, we can rewrite \eqref{45} as $4(E(X-m)_\pm)^2\ge E(X-m)_\pm^2$. Therefore, and because (i) $m$ is a minimizer of $E|X-a|$ in real $a$ and (ii) $EX=0$ is a minimizer of $E(X-a)^2$ in real $a$, we get 
\begin{align*}
	4(E|X|)^2& \ge4(E|X-m|)^2 \\ 
	&=4(E(X-m)_+ +E(X-m)_-)^2 \\ 
	&\ge4(E(X-m)_+)^2+4(E(X-m)_-)^2 \\ 
	&\ge E(X-m)_+^2+E(X-m)_-^2 \\ 
	&=E(X-m)^2\ge EX^2=1. 
\end{align*}
So, $2E|X|\ge1$. 
\end{proof}

\begin{lemma}\label{lem:a,b}
For all real $\de>0$, $b\ge0$, and $x\ge0$, 
\begin{equation}\label{eq:a,b}
	\min(\de,x)\ge r_{\de,b}(x):=a(\de,b)\Big(\frac{x^2}2-b\frac{x^3}3\Big),
\end{equation}
where 
\begin{equation}\label{eq:a}
	a(\de,b):=\min\Big(\frac{16 b}{3},\,6 b^2\de\Big)
	=\left\{
	\begin{alignedat}{2}
 &\frac{16 b}{3} &&\text{\quad if \;} b\ge\frac{8}{9 \de }, \\
 &6 b^2 \de  &&\text{\quad if \;}  b\le \frac{8}{9 \de }.
\end{alignedat}
\right.
\end{equation}
Moreover, for each pair $(\de,b)\in(0,\infty)\times[0,\infty)$, $a(b,\de)$ is the best (that is, the largest) constant factor in \eqref{eq:a,b}.
\end{lemma}

\begin{proof} 
First, let us verify \eqref{eq:a,b}. 

The case $b=0$ is obvious. So, without loss of generality (wlog) $b>0$. Then the maximum of $\frac{x^2}2-b\frac{x^3}3$ in $x\ge0$ is $\frac1{6b^2}$, so that for all $x\ge0$ we have 
\begin{equation*}
r_{\de,b}(x)\le a(\de,b)\frac1{6b^2}\le 6b^2\de \frac1{6b^2}=\de. 	
\end{equation*}
 
The case $x=0$ is obvious. So, wlog $x>0$. Then the maximum of $\frac1x\big(\frac{x^2}2-b\frac{x^3}3\big)$ in $x>0$ is $\frac3{16b}$, so that 
\begin{equation*}
	r_{\de,b}(x)\le a(\de,b)\frac3{16b}\,x\le \frac{16b}3\frac3{16b}\,x=x.
\end{equation*}

Thus, for all real $\de>0$, $b>0$, and $x>0$ we have $r_{\de,b}(x)\le\de$ and \break 
$r_{\de,b}(x)\le x$, so that \eqref{eq:a,b} follows. 

As for the optimality of the constant factor $a(\de,b)$, note that \eqref{eq:a,b} turns into an equality if (i) $b\le \frac{8}{9 \de }$ and $x=\frac1b$ or (ii) $b\ge \frac{8}{9 \de }$ and $x=\frac3{4b}$. 
\end{proof}

\begin{lemma}\label{lem:decr}
Take any real $\de>0$ and any $X\in L$ such that the p.d.f.\ $f$ of $X$ is nonincreasing on $[0,\infty)$. Then 
\begin{equation*}
	P(0<X<\de)\ge p_1(\de):=
	\left\{
	\begin{alignedat}{2}
 &\frac{\de }{72} &&\text{\quad if \;} 0<\de \le \frac{16}{3}, \\
 &\frac{32 (9 \de -32)}{243 \de ^2} &&\text{\quad if \;} \frac{16}{3}\le\de \le \frac{64}{9}, \\
 &\frac{1}{12} &&\text{\quad if \;} \de\ge\frac{64}{9}. 
\end{alignedat}
\right. 
\end{equation*} 
\end{lemma}

\begin{proof} Wlog the p.d.f.\ $f$ is left continuous. 
Let $\nu$ be the (nonnegative) Lebesgue--Stieltjes measure over the interval $[0,\infty)$ defined by the formula 
\begin{equation*}
	\nu([x,\infty))=f(x)
\end{equation*}
for real $x\ge0$. Then 
\begin{equation}\label{eq:P}
\begin{aligned}
	P(0<X<\de)&=\int_0^\de dx\,f(x) \\ 
&=\int_0^\de dx\,\int_{[x,\infty)}\nu(dy) \\ 
&=\int_{[0,\infty)}\nu(dy) \int_0^{\min(\de,y)} dx\, \\ 
&=\int_{[0,\infty)}\nu(dy)\,\min(\de,y),
\end{aligned}	
\end{equation}

\begin{equation}\label{eq:EX_+}
\begin{aligned}
	EX_+&=\int_0^\infty dx\,xf(x) \\ 
&=\int_0^\infty dx\,x\int_{[x,\infty)}\nu(dy) \\ 
&=\int_{[0,\infty)}\nu(dy) \int_0^y dx\,x \\ 
&=\int_{[0,\infty)}\nu(dy)\,\frac{y^2}2,
\end{aligned}	
\end{equation}

\begin{equation}\label{eq:EX_-}
\begin{aligned}
	EX_+^2&=\int_0^\infty dx\,x^2f(x) \\ 
&=\int_0^\infty dx\,x^2\int_{[x,\infty)}\nu(dy) \\ 
&=\int_{[0,\infty)}\nu(dy) \int_0^y dx\,x^2 \\ 
&=\int_{[0,\infty)}\nu(dy)\,\frac{y^3}3. 
\end{aligned}	
\end{equation}

It follows from \eqref{eq:P}, \eqref{eq:EX_+}, \eqref{eq:EX_-}, and Lemma~\ref{lem:a,b} that 
\begin{equation}\label{eq:P ge}
	P(0<X<\de)\ge a(\de,b)(EX_+ - b\, EX_+^2). 
\end{equation}
By Lemma~\ref{lem:E|X|>1/2} and the condition $EX=0$, 
\begin{equation*}
	1/2\le E|X|=EX_+ + EX_-=2EX_+,
\end{equation*}
whence $EX_+\ge1/4$. Also, $EX_+^2\le EX^2=1$. So, in view of \eqref{eq:P ge}, for all real $\de\ge0$ and $b\ge0$,
\begin{equation}\label{eq:P gee}
	P(0<X<\de)\ge p_2(\de,b):=a(\de,b)\Big(\frac{1}{4}-b\Big)
	=\left\{
	\begin{alignedat}{2}
 &\frac{16}{3} b \Big(\frac{1}{4}-b\Big) && \text{\quad if } b\ge\frac{8}{9 \de } \\
 &6 \Big(\frac{1}{4}-b\Big) b^2 \delta  && \text{\quad  if } b\le\frac{8}{9 \de }.
\end{alignedat}
\right. 
\end{equation}
So, $p_2(\de,b)$ is piecewise polynomial in $b$, with the two polynomial pieces of degrees $2$ and $3$. Therefore, it is straightforward but somewhat tedious to maximize $p_2(\de,b)$ in $b\ge0$, to get 
\begin{equation*}
	\max_{b\ge0}p_2(\de,b)=p_1(\de). 
\end{equation*}
However, to complete the proof of Lemma~\ref{lem:decr}, it is enough to note that 
\begin{equation*}
	p_1(\de)=p_2(\de,b_\de),
\end{equation*}
where 
\begin{equation*}
	b_\de:=
\left\{
	\begin{alignedat}{2}
 &\frac16 && \text{\quad if } \de\le\frac{16}3 \\
 &\frac8{9\de}  && \text{\quad  if } \frac{16}3\le\de\le\frac{64}9, \\ 
 &\frac18  && \text{\quad  if } \de\ge\frac{64}9;  
\end{alignedat}
\right. 
\end{equation*}
concerning the piecewise expression of $p_2(\de,b)$ in \eqref{eq:P gee}, note also that $b_\de\le\frac8{9\de}$ if $\de\le\frac{16}3$, and $b_\de\ge\frac8{9\de}$ if $\de\ge\frac{64}9$. 
\end{proof}

By the $x\leftrightarrow-x$ reflection, from Lemma~\ref{lem:decr} we get 
\begin{corollary}\label{cor:incr}
Take any real $\de>0$ and any $X\in L$ such that the p.d.f.\ $f$ of $X$ is nondecreasing on $(-\infty,0]$. Then 
\begin{equation*}
	P(-\de<X<0)\ge p_1(\de).
\end{equation*}   
\end{corollary}

We can now complete the proof of Theorem~\ref{th:}. 

\begin{proof}[Proof of Theorem~\ref{th:}]
Note that 
for all real $\de>0$  
\begin{equation*}
r_1(\de):=\frac{p_1(\de)}{p(\de)}
=
\left\{
	\begin{alignedat}{2}
& 1+c \de  &&\text{\quad if } \de\le\frac{16}{3}, \\
& \frac{256 (9\de-32) (1+c\de)}{27 \de^3} &&\text{\quad if } \frac{16}{3}\le\de\le\frac{64}{9}, \\ 
& 6 (c +1/\de) &&\text{\quad if } \de\ge\frac{64}{9}. 
\end{alignedat}
\right.
\end{equation*}
Clearly, $r_1(\de)\ge1$ if $\de\le\frac{16}{3}$ or $\de\ge\frac{64}{9}$. In the remaining case, when $\frac{16}{3}\le\de\le\frac{64}{9}$, the second derivative in $\de$ of the rational expression for  $r_1(\de)$ is  
\begin{equation*}
	\frac{512 \left(3 c \de ^2+(9-32 c) \de -64\right)}{9 \de ^5}\le0. 
\end{equation*}
So, the minimum of the rational expression for  $r_1(\de)$ in $\de$ in the interval $[\frac{16}{3},\frac{64}{9}]$ is attained at an endpoint of this interval. So, $r_1(\de)\ge1$ for all real $\de>0$. 

This proves Theorem~\ref{th:} in the case when the p.d.f.\ $f$ of $X$ is nonincreasing on $[0,\infty)$. 

It remains to consider the case when $f$ is not nonincreasing on $[0,\infty)$. Then $f$ is nondecreasing on $(-\infty,0]$. So, by Corollary~\ref{cor:incr}, 
\begin{equation*}
	f(0)=\lim_{\de\downarrow0}\frac{P(-\de<X<0)}\de\ge
	\lim_{\de\downarrow0}\frac{p_1(\de)}\de=\frac1{72}. 
\end{equation*}

Also, since $f$ is not nonincreasing on $[0,\infty)$, there is some real $u>0$ such that $f$ is nondecreasing on the interval $[0,u]$ and nonincreasing on the interval $[u,\infty)$. In particular, $f\ge f(0)\ge\frac1{72}$ on the interval $[0,u]$. 

So, if $u\ge\de$, then $f\ge f(0)\ge\frac1{72}$ on the interval $[0,\de]$ and hence
\begin{equation*}
	P(0<X<\de)=\int_0^\de dx\,f(x)\ge\frac\de{72}\ge p(\de).
\end{equation*}
So, wlog 
\begin{equation}\label{eq:u<de}
	0<u<\de. 
\end{equation}
Note that, to get \eqref{eq:P ge}, we did not use any conditions on the r.v.\ $X$ except that its p.d.f.\ $f_X$ be nonincreasing on $[0,\infty)$. Moreover, to get \eqref{eq:P gee} from \eqref{eq:P ge}, we only used the conditions $EX_+\ge1/4$ and $EX_+^2\le1$. So, 
for all real $b\ge0$, noting that the p.d.f.\ of the r.v.\ $X-u$ is nonincreasing on $[0,\infty)$, 
$E(X-u)_+\ge EX_+ -u\ge1/4-u$, and $E(X-u)_+^2\le EX_+^2\le1$, we similarly get 
\begin{align*}
	P(u<X<\de)&=P(0<X-u<\de-u) \\ 
	&\ge p_{2;u}(\de,b):=a(\de-u,b)(1/4-u - b) \\ 
	&=\left\{
	\begin{alignedat}{2}
	&q_1(b)&&\text{ if }b\ge b_u, \\ 
	&q_2(b)&&\text{ if }b\le b_u,  
	\end{alignedat}
	\right.
\end{align*}
where 
\begin{equation*}
	q_1(b):=\frac{16}3\,b\Big(\frac14-u-b\Big),\quad 
	q_2(b):=6\,b^2\Big(\frac14-u-b\Big)(\de-u), 
\end{equation*}
\begin{equation*}
	b_u:=\frac{8}{9(\de-u)}.  
\end{equation*}

It follows that 
\begin{align*}
	P(0<X<\de)&=P(0<X<u)+P(u<X<\de) \\ 
	&\ge \frac u{72}+p_{2;u}(\de,b). 
\end{align*}

Let us now maximize $p_{2;u}(\de,b)$ in $b\ge0$; the latter condition on $b$ will be henceforth assumed by default. In view of \eqref{eq:u<de}, $q_1(b)\le0$ and $q_2(b)\le0$ if $u\ge\frac14$, and $q_1(b)=q_2(b)=0$ if $b=0$. So, 
\begin{equation*}
	\max_{b\ge0}p_{2;u}(\de,b)=0\quad\text{if }u\ge\frac14. 
\end{equation*}

It remains to maximize $p_{2;u}(\de,b)$ in $b\ge0$ assuming that $0<u<\frac14$. Then 
\begin{equation*}
	\max_{b\ge0}q_1(b)=q_1(b_{1;u}),\quad\text{where }b_{1;u}:=\frac12\Big(\frac14-u\Big),
\end{equation*}
\begin{equation*}
	\max_{b\ge0}q_2(b)=q_2(b_{2;u}),\quad\text{where }b_{2;u}:=\frac23\Big(\frac14-u\Big), 
\end{equation*}
\begin{equation*}
	b_{1;u}\ge b_u\iff h\ge r_1(u):=\frac{64}9 \frac1{1-4u},
\end{equation*}
\begin{equation*}
	b_{2;u}\ge b_u\iff h\ge r_2(u):=\frac{16}3 \frac1{1-4u}, 
\end{equation*}
where 
\begin{equation*}
	h:=\de-u>0 
\end{equation*}
(and the latter two displayed logical equivalences hold if each of the entries of $\ge$ there is replaced by $\le$). 
Note also that $r_1\ge r_2$. 

Collecting the pieces, we see that 
\begin{align}
	&\max_{b\ge0}p_{2;u}(\de,b) \notag \\ 
	&=p_{1;u}(\de) \notag \\ 
	&:=\left\{
\begin{alignedat}{2}
& q_2(b_{2;u})=\frac{1}{72} (1-4 u)^3 h &&\text{ if }u<\frac14\ \&\ h\le r_2(u), \\ 
& q_1(b_u)=q_2(b_u)
=\frac{128}{27h}\Big(\frac{1}{4}-u-\frac{8}{9h}\Big)
	&&\text{ if }u<\frac14\ \&\ r_2(u)\le h\le r_1(u), \\ 
&q_1(b_{1;u})=\frac{1}{12} (1-4 u)^2 &&\text{ if }u<\frac14\ \&\ h\ge r_1(u), \\ 
&0 &&\text{ if }u\ge\frac14.
\end{alignedat}
	\right. \label{eq:p_2;u}
\end{align}

So, 
\begin{equation}\label{eq:ge}
	P(0<X<\de)\ge p_u(\de):=\frac u{72}+p_{1;u}(\de). 
\end{equation}
It remains to show that 
\begin{equation}\label{eq:p_u>p}
p_u(\de)\ge p(\de) 
\end{equation}
if $0<u<\de<\infty$. 

This follows immediately from the four lemmas below. 

\begin{lemma}\label{lem:1}
If $0<u<\frac14$, $0<h\le r_2(u)$, and $\de=u+h$, then (cf.\ \eqref{eq:ge}, the first of the four lines in \eqref{eq:p_2;u}, and \eqref{eq:})
\begin{equation*}
	d_1:=\frac u{72}+\frac{1}{72} (1-4 u)^3 h-\frac{\de}{72(1+c\de)}\ge0. 
\end{equation*}
\end{lemma}

\begin{proof}[Proof of Lemma~\ref{lem:1}]
Note that 
\begin{align*}
	D_1(h)& := 72 (100 + 419 (h + u)) d_1 \\ 
	&=419 (1-4 u)^3 h^2 -2 u (13408 u^3-6856 u^2+114 u+181) h +419u^2
\end{align*}
is a strictly convex polynomial in $h$, whose discriminant 
\begin{equation*}
	16 u^2 \left(16 u^2-12 u+3\right) \left(2808976 u^4-765932 u^3-318917
   u^2+131400 u-11900\right)
\end{equation*}
is $\le0$ and hence $D_1(h)\ge0$ for all real $h$ if $u\le24/100$. On the other hand, if $u\ge24/100$, then the critical value of $h$ for this polynomial $D_1(h)$ is 
\begin{equation*}
	\frac{u \left(13408 u^3-6856 u^2+114 u+181\right)}{419 (1-4 u)^3}\le0,
\end{equation*}
so that $D_1(h)\ge D_1(0)=419 u^2\ge0$. So, $D_1(h)\ge0$ and hence $d_1\ge0$ for all $u\in(0,\frac14)$ and all real $h>0$. (The condition $h\le r_2(u)$ was not needed or used in this proof.)
\end{proof}

\begin{remark}\label{rem:}
The sign pattern of a polynomial in $\R[x]$ can be determined using Sturm's theorem -- see e.g.\ \cite[p.\ 244]{vanderWaerden}. Especially when the degree of the polynomial is rather small, one can also use calculus to determine the convexity and monotonicity patterns.  
In principle, all this can be done by hand. However, this is much more easy to do using any computer algebra system. 
\end{remark}

\begin{lemma}\label{lem:2}
If $0<u<\frac14$, $r_2(u)\le h\le r_1(u)$, and $\de=u+h$, then (cf.\ \eqref{eq:ge}, the second of the four lines in \eqref{eq:p_2;u}, and \eqref{eq:})
\begin{equation*}
	d_2:=\frac u{72}+\frac{128}{27h}\Big(\frac{1}{4}-u-\frac{8}{9h}\Big)-\frac{\de}{72(1+c\de)}\ge0. 
\end{equation*}
\end{lemma}

\begin{proof}[Proof of Lemma~\ref{lem:2}]
Note that 
\begin{align*}
	D_2& := 216 9 h^2 (100 + 419 (h + u))\, d_2 \\ 
	&=27 h^3 (419 u-100)+3771 h^2 \left(3 u^2-1024 u+256\right) \\ 
	&\quad -256 h
   \left(15084 u^2-171 u+12508\right)-8192 (419 u+100)  
\end{align*}
is a polynomial in $u,h$. Consider the partial derivatives of $D_2$ in $u$ and in $h$: 
\begin{align*}
	D_{21}&:=\partial_u D_2 \\ 
	&=11313 h^3+22626 h^2 u-3861504 h^2-7723008 h u+43776 h-3432448, \\ 
		D_{22}&:=\partial_h D_2 \\ 
	&=33939 h^2 u-8100 h^2+22626 h u^2-7723008 h u+1930752 h-3861504 u^2 \\ 
	&\phantom{33939 h^2 u-8100 h^2+22626 h u^2-7723008 h u\qquad}+43776 u-3202048.  	
\end{align*}
The resultants of the polynomials $D_{21}$ and $D_{22}$ in $u,h$ with respect to $u$ and $h$ are 
\begin{align*}
	R_1(h)&:=-383951907 h^7+360127950804 h^6-111582387392256 h^5 \\ 
	&+11376013014678528
   h^4+7495868032745472 h^3-5874950492651520 h^2 \\ 
   &+17672548909056
   h-3016115013812224,
\end{align*}
\begin{align*}
	R_2(u)&:=-1807585595653280832 u^7-847701988664287064715  
   u^6 \\ 
   &-78523365776753581860762 u^5+44892310928843299875696
   u^4 \\ 
   &+43252111127403174064608 u^3-35398357322310505259136
   u^2 \\ 
   &+29165745137518115033088 u-5192266514139579318272. 
\end{align*}

Note that $r_2(u)>\frac{16}3$, so that the condition $r_2(u)\le h\le r_1(u)$ in Lemma~\ref{lem:2} implies $h>\frac{16}3$.

There are three real roots of $R_1(h)$ that are $>\frac{16}3$: 
\begin{equation*}
	(h_1,h_2,h_3)\approx(255.934, 341.325, 341.342). 
\end{equation*}
There is only one root of $R_2(u)$ in the interval $(0,\frac14)$, $u_1\approx0.218271$. 
Therefore (see e.g.\ \cite[top of p.\ 104]{vanderWaerden}) $D_{21}=0=D_{22}$ only if $(u,h)=(u_1,h_j)$ for some $j=1,2,3$. 

On the other hand, for all $j=1,2,3$, 
\begin{equation*}
	D_2|_{u=u_1,\;h=h_j}>3\times10^9>0.
\end{equation*}
So, $D_2>0$ at all critical points point $(u,h)$ of $D_2$ such that $0<u<\frac14$ and $r_2(u)<h<r_1(u)$.

Also, $r_1(u)>r_2(u)\to\infty$ as $u\uparrow\frac14$ and hence 
\begin{equation*}
	\lim_{u\uparrow\frac14}\ \min_{h\in[r_2(u),r_1(u)]}\frac{D_2}{h^3}
	=27(419\times\tfrac14-100)>0, 
\end{equation*}
so that $D_3\to\infty>0$ uniformly in $h\in[r_2(u),r_1(u)]$ as $\uparrow\frac14$, which shows that $D_2>0$ for some $u_0\in(0,\frac14)$, all $u\in(u_0,\frac14)$, and all $h\in[r_2(u),r_1(u)]$.  

Therefore, to complete the proof of Lemma~\ref{lem:2}, it remains to show that $D_2\ge0$ on the boundary of the (unbounded) set 
\begin{equation*}
	G:=\{(u,h)\colon 0<u<\tfrac14,\ r_2(u)\le h\le r_1(u)\}. 
\end{equation*}
To do this, note first that $r_2(0)=\frac{16}{3}$, $r_1(0)=\frac{64}{9}$, and 
\begin{equation*}
	D_2|_{u=0}=4(-675 h^3+241344 h^2-800512 h-204800)>0
\end{equation*}
if $r_2(0)\le h\le r_1(0)$. 
It remains to note that 
\begin{align*}
	&D_2|_{h=r_1(u)} \\ 
	&=-\frac{4096}{27(1 - 4 u)^3} \left(1448064 u^4-725364 u^3-2565795 u^2+1302526
   u-159896\right)\ge0,
\end{align*}
\begin{align*}
	&D_2|_{h=r_2(u)} \\ 
	&=-\frac{256}{3(1 - 4 u)^3} \left(1287168 u^4-643092 u^3-1709051 u^2+875488
   u-107264\right)\ge0
\end{align*}
if $0<u<\tfrac14$. 
\end{proof}

\begin{lemma}\label{lem:3}
If $0<u<\frac14$, $h\ge r_1(u)$, and $\de=u+h$, then (cf.\ \eqref{eq:ge}, the third of the four lines in \eqref{eq:p_2;u}, and \eqref{eq:})
\begin{equation*}
	d_3:=\frac u{72}+\frac{1}{12} (1-4 u)^2-\frac{\de}{72(1+c\de)}\ge0. 
\end{equation*}
\end{lemma}

\begin{proof}[Proof of Lemma~\ref{lem:3}]
Note that 
\begin{align*}
	D_3(\de)& := 72 (100 + 419 \de) d_3 \\ 
	&=(40224 u^2-19693 u+2414)\de+100 \left(96 u^2-47u+6\right)
\end{align*}
is a polynomial of degree $1$ in $\de$ with coefficients that are positive for all real $u$. 
So, for all $\de>0$ and all real $u$, we have $D_3(\de)\ge0$ and hence $d_3\ge0$. (The conditions $0<u<\frac14$ and $h\ge r_1(u)$ were not needed or used in this proof.)
\end{proof}

\begin{lemma}\label{lem:4}
If $u\ge\frac14$, then (cf.\ \eqref{eq:ge}, the fourth of the four lines in \eqref{eq:p_2;u}, and \eqref{eq:})
\begin{equation*}
	d_4:=\frac u{72}-\frac{\de}{72(1+c\de)}\ge0. 
\end{equation*}
\end{lemma}

\begin{proof}[Proof of Lemma~\ref{lem:4}]
Note that, for $u\ge\frac14$ and $\de>0$,
\begin{align*}
	d_4\ge\frac{1/4}{72}-\frac{\de}{72(0+c\de)}
	=\frac{1}{72\times4}-\frac{1}{72c}\ge0. 
\end{align*}
Lemma~\ref{lem:4} is proved. 
\end{proof}

This completes the proof of Theorem~\ref{th:}. 
\end{proof}







\bibliographystyle{abbrv}


\bibliography
{C:/Users/ipinelis/Documents/pCloudSync/mtu_pCloud_02-02-17/bib_files/citations04-02-21}

\end{document}